\documentclass[10pt,a4paper]{article}
\usepackage{graphicx}
\usepackage{lineno,url}
\usepackage{algorithm}
\usepackage{algorithmic}
\usepackage{amscd,amsmath,amssymb,amsthm,amsfonts}
\usepackage{authblk}

\newtheorem{prop}{\bf Proposition}[section]
\newtheorem{thm}[prop]{\bf Theorem}
\newtheorem{lem}[prop]{\bf Lemma}

\theoremstyle{definition}

\newtheorem{defi}[prop]{Definition}
\newtheorem{rmk}[prop]{Remark}

\newtheorem{condi}[prop]{Condition}

\newcommand{\map}{\rightarrow}

\newcommand{\BP}{\mathbb{P}}

\newcommand{\QD}{\mathrm{QD}}
\newcommand{\Sing}{\mathrm{Sing}}
\newcommand{\PMF}{\mathcal{PMF}}
\newcommand{\bT}{\bar{\mathcal{T}}}
\newcommand{\bC}{\bar{\mathcal{C}}}

\title{Fibered commensurability and arithmeticity of random mapping tori}
\author{Hidetoshi Masai}
\date{}

\begin{document}
\maketitle
\begin{abstract}
We consider a random walk on the mapping class group of a surface of finite type.
We assume that the random walk is determined by a probability measure whose
support is finite and generates a non-elementary subgroup $H$.
We further assume that $H$ is not consisting only of lifts with respect to any one covering.
Then we prove that the probability that such a random walk gives a non-minimal mapping class in its fibered commensurability class
decays exponentially.
As an application of the minimality, we prove that for the case where a surface has at least one puncture,
the probability that
a random walk gives mapping classes with arithmetic mapping tori decays exponentially.
We also prove that a random walk gives rise to asymmetric mapping tori with exponentially high probability for closed case.

\end{abstract}

\section{Introduction} 
Let $S$ be an orientable surface of finite type $(g,n)$, where $g$ is the genus and $n$ is the number of punctures.
We consider a random walk on the mapping class group $G:=\mathrm{Mod}(S)$ which is determined by a probability 
measures on $G$ whose support generates a non-elementary subgroups.
It has been shown that such a random walk gives rise to pseudo-Anosov elements with asymptotic probability one 
\cite{Kow,Mah2011,Mah2012, Rivin}.
Let $\mu$ be a probability measure on $G$.
A subset $A\subset G$ is said to be exponentially small (with respect to $\mu$) if the probability that 
the random walk determined by $\mu$ visits $A$ decays exponentially with the number of steps.
A subset is called exponentially large (with respect to $\mu$) if its complement is exponentially small.
The work of Maher \cite{Mah2012} can be stated as ``the set of pseudo-Anosov elements is exponentially large".
In this paper, we consider fibered commensurability, a notion introduced by Calegari-Sun-Wang \cite{CSW},
of random mapping classes.
Roughly, a mapping class $\phi$ is said to cover another mapping class $\varphi$ 
if $\phi$ is a power of some lift of $\varphi$ with respect to some finite covering of underlying surfaces.
The commensurability with respect to this covering relation is called fibered commensurability.
Each commensurability class enjoys an order by the covering relation.
It has been shown \cite{CSW, Masa} that for pseudo-Anosov case, each commensurability class contains 
a unique minimal (orbifold) element (see Theorem \ref{thm.uniqueminimal}).
Our aim is to prove that the set of minimal elements is exponentially large 
with respect to any measure which satisfies a suitable condition (Condition \ref{condi.munotlift}).
As an application of the minimality, we also show a result on arithmeticity of random mapping tori. 
By using random walks on $G$, we may generate randomly 3-manifolds by taking mapping tori.
The work of Thurston \cite{Thu} together with \cite{Mah2012} shows that the set of mapping classes with hyperbolic
mapping tori is exponentially large.
A cusped hyperbolic 3-manifold is called arithmetic if it is commensurable to a Bianchi orbifold (see \S \ref{sec.arithmetic}).
Several distinguished hyperbolic 3-manifolds, for example the complement of the figure eight knot or the Whitehead link, are known
to be arithmetic.
However, a ``generic" hyperbolic 3-manifold is believed to be non-arithmetic.
The minimality of random mapping classes together with the work by Bowditch-Maclachlan-Reid \cite{BMR}
enables us to prove that the set of mapping classes with arithmetic mapping tori is
exponentially small if $S$ has at least one puncture. 
We also prove that the set of mapping classes with asymmetric mapping tori is exponentially large for closed case.

The paper is organized as follows.
In \S \ref{sec.pre}, we prepare several definitions and facts about random walks on groups and mapping class groups.
Note that to prove that a given mapping class $\phi$ is minimal, 
it suffices to show that $\phi$ is primitive and not symmetric.
In \S \ref{sec.primitive} we prove the primitivity of random mapping classes.
\begin{thm}\label{thm.primitive}
Let $\mu$ be a probability measure on $G$ whose support is finite and generates a non-elementary subgroup.
Then the set of primitive elements in $G$ is exponentially large with respect to $\mu$.
\end{thm}
Next, we prove that random mapping classes are not symmetric in \S \ref{sec.notlift}.
We call a mapping class {\em symmetric} if it is a lift with respect to some finite covering $\pi:S\map S'$.
We need further assumption for the measure $\mu$ to avoid the case that there is some finite covering 
$\pi:S\map S'$ such that every element in the support of $\mu$ is a lift of a mapping class on $S'$.
Let $\mathcal{PMF}(S)$ denote the set of projective measured foliations on $S$, where in the case of orbifolds, we consider
the one for the surface we get by puncturing the orbifold points.
Each covering $\pi:S\map S'$ determines a map $\Pi:\mathcal{PMF}(S')\map \mathcal{PMF}(S)$ 
so that $a\in\Pi(\PMF(S'))$ if and only if $\pi(a)\in\PMF(S')$.
Let $\mathrm{gr}$ (resp. $\mathrm{sgr}(\mu)$) denote the group (resp. semigroup) generated by the support of $\mu$.
The condition for the measure $\mu$ which we need is the following.
\begin{condi}\label{condi.munotlift}
The support is finite, and generates a non-elementary subgroup of $G$.
Moreover, for any (possibly orbifold) covering $\pi:S\map S'$, $\mathrm{sgr}(\mu)$ contains 
a pseudo-Anosov element 
whose fixed points set is disjoint from $\Pi(\mathcal{PMF}(S'))$.
\end{condi}
In \S \ref{sec.notlift}, we prove:
\begin{thm}\label{thm.notlift}
Let $\mu$ be a probability measure on $G$ which satisfies Condition \ref{condi.munotlift}.
Then the set of symmetric elements
is exponentially small with respect to $\mu$.
\end{thm}
Putting Theorem \ref{thm.primitive} and \ref{thm.notlift} together, we have:
\begin{thm}\label{thm.minimal}
Let $\mu$ be a probability measure on $G$ which satisfies Condition \ref{condi.munotlift}.
Then the set of minimal elements in their fibered commensurability class is exponentially large with respect to $\mu$.
\end{thm}
Finally in \S \ref{sec.arithmetic}, we prove the following theorem.
\begin{thm}\label{thm.arithmetic}
Suppose that $S$ has at least one puncture.
Let $\mu$ be a probability measure on $G$ which satisfies Condition \ref{condi.munotlift}.
Then the set of mapping classes with arithmetic mapping tori is exponentially small with respect to $\mu$.
\end{thm}
In \S \ref{sec.asymmetric}, it is proved that closed random mapping tori are asymmetric.

\section{Preliminary}\label{sec.pre}
In this section, we summarize several definitions and facts that we use throughout the paper. 
Interested readers may refer to several papers regarding to random walks on the mapping class groups 
(for example \cite{KM, Mah2011}) in which
there are detailed expositions of basic theory of both random walks and mapping class groups.
\subsection{Random walks on groups}
We recall the definitions and terminologies of random walks.
See \cite{Wo} for more details about random walks on groups.
Let $G$ be a countable group. 
A (possibly infinite) matrix $\mathbb{P} = (p_{g,h})_{g,h\in G}$ is called {\em stochastic} if every element is non-negative and
$$\sum_{h\in G} p_{g,h} = 1$$
for all $g\in G$.
For a given probability measure $\mu$ on $G$,
by putting $p_{g,h} = \mu(g^{-1}h)$, we have a stochastic matrix $\mathbb{P}_\mu = (p_{g,h})_{g,h\in G}$.
Let $P_n$ denote the probability measure on $(G^n,2^{G^n})$ defined by
$$P_n(A) = \sum_{(g_1,\dots, g_{n})\in A} p_{\mathrm{id}, g_1}p_{g_1, g_2}\cdots p_{g_{n-1},g_{n}} \text{ for } A\in 2^{G^n}.$$
Note that by definition, we have $P_{n+1}(A\times G) = P_n(A)$ for any $A\in 2^{G^n}$.
Let $\mathcal{B}(G^\mathbb{N})$ denote the $\sigma$-algebra generated by cylinder sets, where
a cylinder set is a subset defined as 
$$\{\omega = (\omega_n)_{n\in\mathbb{N}}\in G^\mathbb{N}\mid (\omega_1, \dots, \omega_n) \in A\}$$
for some $A\subset G^n$.
Then by the Kolmogorov extension theorem, there exists a unique measure $P$ on $(G^\mathbb{N}, \mathcal{B}(G^\mathbb{N}))$ which satisfies
$$ P(A\times G^\mathbb{N}) = P_n(A) \text{ for all } n\in \mathbb{N}, \text{ and } A\in 2^{G^n}.$$
For $\omega = (\omega_n)\in G^\mathbb{N}$, we define $G$-valued random variables $X_n$ 
on $(G^{\mathbb{N}},\mathcal{B}(G^\mathbb{N}))$ by $X_n(\omega) = \omega_n$.
Thus we have a stochastic process $\{X_n\}_{n\in \mathbb{N}}$ which is a Markov chain with the transition matrix $\mathbb{P}_{\mu}$.
We call this Markov chain $\{X_n\}_{n\in\mathbb{N}}$ the {\em random walk} determined by $\mu$.

Let us fix a probability measure $\mu$ and the random walk determined by $\mu$.
Each element $(\omega_n)_{n\in\mathbb{N}}\in G^{\mathbb{N}}$ is called a {\em sample path}.
Let $A\subset G$. 
By abbreviation of notations, we write $\BP(\omega_n \in A)$ to mean $P(G^{n-1}\times A \times G^\mathbb{N})$. 
A subset $A\subset G$ is called {\em exponentially small} (with respect to $\mu$) 
if there exist $c<1, K>0$ which depend only on $\mu$ and $A$
such that $\BP(\omega_n \in A)<Kc^n$.
A subset is called {\em exponentially large} (with respect to $\mu$) if its complement is exponentially small.
Let $Q$ be a property for elements in $G$.
We say that the random walks determined by $\mu$ has property $Q$ {\em with exponentially high probability} if
$S_Q:= \{g\in G\mid g \text{ is } Q\}$ is exponentially large.
It can be readily seen that if $A,B\subset G$ are both exponentially small (resp. large), then so is $A\cup B$ (resp. $A\cap B$).

\subsection{Mapping class groups and curve graphs}
For more details about topics in this subsection, one may refer to the books \cite{Bow2, FM}.
Let $S:= S_{g,n}$ be an orientable surface of finite type $(g,n)$ where $g$ is the genus and $n$ is the number of punctures.
In this paper, we always suppose $3g-3+n>0$ unless otherwise stated.
The {\em mapping class group} $\mathrm{Mod}(S)$ is the group of isotopy classes of orientation preserving automorphisms on $S$.
A mapping class is called {\em pseudo-Anosov} if it is aperiodic and has no fixed 1-dimensional submanifold of $S$.
Thurston \cite{Thu2} 
showed that each pseudo-Anosov mapping class has exactly two fixed points  $\mathcal{F}_s, \mathcal{F}_u$
in the space $\mathcal{PMF}(S)$ of projective measured foliations.
A subgroup of $\mathrm{Mod}(S)$ 
is called {\em non-elementary} if it contains two pseudo-Anosov mapping classes with disjoint fixed points in
$\mathcal{PMF}(S)$.

The {\em curve graph} $\mathcal{C}(S)$ of $S$ is a graph whose vertices consist of isotopy classes of simple closed curves,
and two vertices are connected by an edge if the corresponding curves can be disjointly represented on $S$.
By giving length $1$ to every edge, the curve graph enjoys a metric $d_{\mathcal{C}(S)}(\cdot, \cdot)$.
If $S$ is an annulus, then vertices of $\mathcal{C}(S)$ are essential arcs, considered up to isotopy relative to their boundary.
Edges are placed between vertices with representatives having disjoint interiors.

Let $(X, d_X)$ be a metric space.
For a fixed point $p\in X$, the {\em Gromov product}  $(x\cdot x')_p$ of two points $x,x'\in X$ is defined by 
$$(x\cdot x')_p = \frac{1}{2}(d_X(x,p) + d_X(x',p) - d_X(x,x')).$$
Then for $r>0$, a {\em shadow} $S_p(x,r)\subset X$ is defined by
$$S_p(x,r) := \{y\in X\mid (x\cdot y)_p\geq r \}.$$
If we have another metric space $(Y,d_Y)$, 
a map $f:X\map Y$ is said to be {\em $Q$-quasi-isometric} if for any $x,x'\in X$,
$$d_X(x,x')/Q - Q \leq d_Y(f(x), f(x')) \leq Qd_X(x,x') + Q.$$
Such $f$ is called {\em $Q$-quasi-isometry} if it further satisfies that for any $y\in Y$, there exists $x\in X$ such that
$d_Y(y, f(x))<Q$.
Two metric spaces are said to be {\em quasi-isometric} if there is a $Q$-quasi-isometry between the two.
Suppose further that $X$ is a geodesic space.
Then $X$ is called {\em $\delta$-hyperbolic} if every geodesic triangle is {\em $\delta$-thin};
one side of a geodesic triangle is contained in the $\delta$-neighborhood of the other two sides.
$X$ is called {\em hyperbolic} if it is $\delta$-hyperbolic for some $\delta\geq0$.
Two geodesics in $X$ are said to be {\em asymptotic} if they are finite Hausdorff distance apart.
We define {\em the Gromov boundary} as the set of asymptotic classes of geodesics.
Hyperbolicity is invariant under quasi-isometries,
and a quasi-isometry induces a homeomorphism of the Gromov boundaries.
For two points $x,x'$ in a geodesic space $X$, we denote by $[x,x']$ a geodesic connecting $x$ and $x'$.
Note that there can be many such geodesics, and $[x,x']$ is an arbitrarily chosen one.
We suppose that if $a,b\in [x,x']$, then $[a,b]\subset[x,x']$.
\begin{rmk}\label{rmk.shadow}
It is well known that if $X$ is $\delta$-hyperbolic, the Gromov product $(x,x')_p$ is equal to the distance from $p$ to $[x,x']$
up to additive constant $K$ which depends only on $\delta$ (c.f. Lemma \ref{lem.fourpt}).
By this fact, a shadow $S_p(x, r)$ for $x\in X$ and $r>0$ can be (coarse equivalently) regarded as the set of $x'\in X$ such that
every geodesic connecting $p$ and $x'$ passes through a point in the $(d_X(x,p)-r+C)$-neighborhood of $x$ for some 
$C$ depending only on $\delta$.
\end{rmk}

In \cite{MM}, Masur-Minsky proved that the curve graph $\mathcal{C}(S)$ is hyperbolic. 
The mapping class group $G:=\mathrm{Mod}(S)$ acts isometrically on $\mathcal{C}(S)$.
Using this action, by fixing a base point $p\in\mathcal{C}(S)$, $G$ admits a $\delta$-hyperbolic (improper) metric which we denote again by $d_{\mathcal{C}(S)}$;
$$d_{\mathcal{C}(S)}(g,h) = d_{\mathcal{C}(S)}(gp,hp) .$$

\subsection{Commensurability of mapping classes}
In \cite{CSW}, Calegari-Sun-Wang defined commensurability of mapping classes on possibly distinct surfaces as follows.
\begin{defi}[\cite{CSW}]
Let $S_1$ and $S_2$ be orientable surfaces of finite type.
A mapping class $\phi_1\in\mathrm{Mod}(S_1)$ {\em covers} $\phi_2\in\mathrm{Mod}(S_2)$ if
there exists a finite covering $\pi:S_1\map S_2$ and 
$k\in\mathbb{Z}\setminus\{0\}$ such that a lift $\varphi$ of $\phi_2$ with respect to $\pi$ satisfies 
$\varphi^k = \phi_1$.
Two mapping classes are said to be {\em commensurable}
if there exists a mapping class that covers both.
\end{defi}

Since this gives commensurability of the monodromies of fibers on orientable surface bundles over the circle, 
this notion is also called {\em fibered commensurability}.
Commensurability gives rise to an equivalence relation by taking transitive closure. 
We consider conjugacy classes in order to have each commensurability class enjoy an order by covering relation (see \cite{CSW} for a detail).
We call a mapping class {\em minimal} if it is a minimal element with respect to the order in its commensurable class.
By extending our category to the orbifolds and  orbifold automorphisms, for the cases where mapping classes are pseudo-Anosov, we have the following uniqueness of minimal element.
\begin{thm}[\cite{CSW, Masa}]\label{thm.uniqueminimal}
If $\phi\in\mathrm{Mod}(S)$ is pseudo-Anosov, then the commensurability class of $\phi$ contains a unique minimal (orbifold) element.
\end{thm}
Note that a mapping class $\phi$ is minimal if it is {\em primitive} (i.e. if $\varphi^k = \phi$, then $k = 1$ and $\phi = \varphi$, or
$k = -1$ and $\phi = \varphi^{-1}$)
and it is not a lift of any orbifold automorphism.

\section{Random mapping classes are primitive}\label{sec.primitive}
Throughout this section, let us fix an orientable surface $S$ of finite type and denote by 
$G$ the mapping class group $\mathrm{Mod}(S)$.
To prove the primitivity, we consider the action of $G$ on the curve graph $\mathcal{C}(S)$.
We shall fix a base point $p\in\mathcal{C}(S)$.
For $g\in G$, the translate $gp\in\mathcal{C}(S)$ is also denoted by $g$ by abuse of notation.
We abbreviate the distance on $\mathcal{C}(S)$ to $d_\mathcal{C}(\cdot, \cdot)$.
In this section, unless otherwise stated, we consider the random walk determined by 
a probability measure $\mu$ on $G$ with finite support which generates a non-elementary subgroup.

\subsection{Random mapping classes do not (anti-)align}
We first recall the work of Calegari-Maher \cite{CM}.
\begin{defi}
Let $p_0, \dots, p_n$ be points in $\mathcal{C}(S)$ and $\gamma = [p_0, p_n]$. 
A point $y\in\gamma$ is {\em $D$-proximal} (with respect to $p_0, \dots, p_n$)
 if $d_\mathcal{C}(y, p_i)<D$ for some $0\leq i\leq n$.
Let $\gamma_D$ denote the subset of $D$-proximal points of $\gamma$.
\end{defi}
Let $\omega = (\omega_n)$ be a sample path in $G^\mathbb{N}$,
then for large enough $n$, 
Calegari-Maher proved that most part of $[\omega_0, \omega_n]$ should be $D$-proximal with exponentially high probability.
\begin{lem}[{\cite[Lemma 5.14]{CM}}]\label{lem.proximal}
There are constants $C_1, K > 0$ and $c<1$ so that for any $\epsilon > 0$, there is a further constant $D$ depending on
$C_1$ and  $\epsilon$ with the following property.
Let $\gamma:= [\omega_0, \omega_n]$ and $\gamma_D$ denote the set of $D$-proximal points on $\gamma$ with respect to 
$\omega_0, \dots\omega_n\in\mathcal{C}(S)$. Then 
 $$\mathbb{P}((\mathrm{length}(\gamma) \geq C_1n) \wedge 
 (\mathrm{length}(\gamma_D )/\mathrm{length}(\gamma ) \geq 1 - \epsilon)) \geq 1 - Kc^n.$$
\end{lem}
Lemma \ref{lem.proximal} shows that coarsely, any random walk fellow travels with a geodesic connecting the endpoints with
exponentially high probability.

We also recall the work of Maher which shows that each shadow is exponentially small.
\begin{lem}[\cite{Mah2012}]\label{lem.shadow}
There are constants $K > 0$ and $c<1$ such that for any $q \in \mathcal{C}(S)$ and 
any $r$,
$$\mathbb{P}(\omega_n \in S_{1} (q,r))< Kc^r.$$
\end{lem}

Throughout in this section, we suppose that a path in $\mathcal{C}(S)$ is a continuous map $[0,1]\map\mathcal{C}(S)$.
Hence for a given path $\gamma$, $\gamma(0)$ denotes the initial point and $\gamma(1)$ denotes the terminal point.
Two paths $\gamma_1$ and $\gamma_2$ are said to be {\em $D$-aligned} (resp. {\em $D$-anti-aligned}) if 
there exists $h\in G$ such that $d_\mathcal{C}(h\gamma_1(0),\gamma_2(0))<D$ and 
$d_\mathcal{C}(h\gamma_1(1),\gamma_2(1))<D$ 
(resp. $d_\mathcal{C}(h\gamma_1(1),\gamma_2(0))<D$ and $d_\mathcal{C}(h\gamma_1(0),\gamma_2(1))<D$).
Lemma \ref{lem.align} below looks quite similar to {\cite[Lemma 5.26]{CM}}
showing the probability that a random walk has two anti-aligned subpaths decays polynomially.
Lemma \ref{lem.align} shows the probability that a random walk has aligned subpaths decays exponentially.
The order of the decay is exponential since we consider the case that a random walk has aligned subpaths of length 
of linear order (see property (1) of Lemma \ref{lem.align}) while in \cite{CM}, the order was of logarithm.
Although one can prove Lemma \ref{lem.align} by almost the same argument as in \cite{CM}, 
we include a proof for completeness.
Recall that by the work of Bowditch \cite{Bow}, the action of $G$ on $\mathcal{C}(S)$ is {\em acylindrical};
for any $C_1 > 0$, there are constants
$C_2,C_3$ such that for $a,b\in \mathcal{C}(S)$ with $d_{\mathcal{C}}(a,b)\geq C_2$,
there are at most $C_3$ elements $h\in G$ with $d_{\mathcal{C}}(a, ha)\leq C_1$
and $d_{\mathcal{C}}(b, hb)\leq C_1$.
\begin{lem}[c.f. {\cite[Lemma 5.26]{CM}}]\label{lem.align}
Fix $D, M>0$.
Then there is a constant $c_1<1$, $K>0$
such that the following holds. 
Consider the collection of indices $a < a' < b < c < c' < d$
 for which there are geodesics $\alpha\in[\omega_a,\omega_b]$ and $\beta\in[\omega_c,\omega_d]$ with the following properties:
\begin{enumerate}
\item $\mathrm{length}(\alpha)\geq Mn$ and similarly for $\beta$;
\item  there is $t\in[0.1,0.2]$ so that $d_{\mathcal{C}}(\omega_{a'}, \alpha(t))\leq D$
 and $d_{\mathcal{C}}(\omega_{c'}, \beta(t))\leq D$;
\item there is some $h \in G$ so that $d_{\mathcal{C}} (h\alpha(0), \beta(0)) \leq D$, and
$d_{\mathcal{C}} (h\alpha(1), \beta(1)) \leq D$.
\end{enumerate}
The probability that this collection of indices is non-empty is at most $Kc_1^n$.
\end{lem}
\begin{proof}
We first fix $a<a'<b<c<c'$.
To satisfy conditions (2) and (3), we need to have $h\in G$ such that $d_\mathcal{C}(h\omega_a, \omega_c)\leq C_1$ and 
$d_\mathcal{C}(h\omega_{a'}, \omega_{c'})\leq C_1$ for some constant $C_1$ depending only on $D$ and the hyperbolicity constant $\delta$.
Hence, the acylindricity of the action of $G$ on $\mathcal{C}(S)$ implies that if $\alpha = [\omega_a,\omega_b]$ is long enough,
there is a set $A\subset \mathcal{C}(S)$ of at most $C_3$ points so that
$\omega_d$ should be in $D$ neighborhood of some point $x\in A$ where $C_3$ depends only on $D$ and $\delta$.
By Remark \ref{rmk.shadow}, it follows that $\omega_d\in S_{\omega_{c'}}(x, d_\mathcal{C}(\omega_{c'}, x)-C)$
for some $C$ depending only on $\delta$.
Then by Lemma \ref{lem.shadow}, 
the probability that a random walk from $\omega_{c'}$ is in $S_{\omega_{c'}}(x, d_\mathcal{C}(\omega_{c'}, x)-C))$ 
decays exponentially since $d_\mathcal{C}(\omega_{c'}, x)$ is at least $8Mn/10$ by the conditions (1) and (2).
Since the number of elements of $A$ is universally bounded, the probability
that $a<a'<b<c<c'$ will be followed by some $d$ which satisfies (1)-(3) is less than $K'c_2^n$ for some $K'>0$ and $c_2<1$ 
which depend only on $D, \delta$ and $M$ but not
on $n$ and $a<a'<b<c<c'$.
The number of all possible choices of $a<a'<b<c<c'$ is of order $n^5$.
We may find some $K>0$ and $c_1<1$ such that $n^5K'c_2^n<Kc_1^n$.
Thus we complete the proof.
\end{proof}
\begin{rmk}\label{rmk.anti-align}
As shown in \cite{CM}, almost the same argument shows anti-aligned version of Lemma \ref{lem.align}. 
Namely, we may replace the conditions (2) and (3) of Lemma \ref{lem.align} with
\begin{enumerate}
\item[(2)']
	there is $t\in[0.1,0.2]$ so that $d_{\mathcal{C}}(\omega_{a'}, \alpha(1-t))\leq D$
	and $d_{\mathcal{C}}(\omega_{c'}, \beta(t))\leq D$,
 \item[(3)'] 
 	there is some $h \in G$ so that $d_{\mathcal{C}} (h\alpha(0), \beta(1)) \leq D$, and
	$d_{\mathcal{C}} (h\alpha(1), \beta(0)) \leq D$,
\end{enumerate}
to have the probability that we have indices satisfying (1), (2)' and (3)' decays exponentially.
\end{rmk}
\subsection{Proof of Theorem \ref{thm.primitive}.}
For $g\in G$, let $\tau(g)$ denote  the {\em translation length} 
$$\tau(g) := \lim_{n\to\infty} \frac{d_{\mathcal{C}}(g^n(p),p)}{n}$$
of $g$ on the curve graph $\mathcal{C}(S)$.
Maher-Tiozzo proved that the translation length grows linearly \cite{MT}.
\begin{lem}[\cite{MT}]\label{lem.translen}
There exists $L>0$, $K>0$ and $c<1$ which only depends on $S$ and $\mu$ such that
$$\mathbb{P}(\tau(\omega_n)< Ln)<Kc^n.$$
\end{lem}
We first prepare an elementary observation for an action of a group on a $\delta$-hyperbolic space.
\begin{prop}[c.f. {\cite[Lemma 3.3]{Mah2011}}]\label{prop.fellowtravel}
Let $H$ be a group acting isometrically on a $\delta$-hyperbolic space $(Y,d_Y)$ with a base point $x$.
Fix $h\in H$.
Suppose that 
$h$ has a geodesic axis $\alpha$, i.e. a geodesic satisfying $h^n(\alpha)\subset \mathcal{N}_{2\delta}(\alpha)$ 
for all $n\in\mathbb{Z}$ where
$\mathcal{N}_{2\delta}(\alpha)$ denotes the $2\delta$ neighborhood of $\alpha$.
Let $q$ be a nearest point projection of $x$ to $\alpha$.
If $d_Y(q,hq)> 28\delta$, the following holds.
There exist $D_1,D_2\geq 0$ which depend only on $\delta$ such that the geodesic $\gamma = [x,hx]$ can be decomposed into three subsegments $\gamma = \gamma_1\gamma_2\gamma_3$
so that
\begin{itemize}
\item The distance $d_{\mathcal{C}}(\gamma_1(1), q)\leq D_1$ and $d_{\mathcal{C}}(\gamma_3(0), hq)\leq D_1$, and
\item $\gamma_2\subset \mathcal{N}_{D_2}(\alpha)$ and $\mathrm{length}(\gamma_2)\geq d_Y(q,hq)-28\delta$.
\end{itemize}
\end{prop}
\begin{proof}
Any side of a geodesic quadrilateral in a $\delta$-hyperbolic space is in the $2\delta$ neighborhood of the other three sides.
We consider a geodesic quadrilateral whose vertices are $x, q, hq, hx$.
Since $q, hq$ are nearest point projections, if a point $s\in [q,hq]$ is at least $4\delta$ apart from $q$ and $hq$,
then $d_Y(s,\gamma)\leq 2\delta$.
This is because if $d_Y(s,\gamma)> 2\delta$, then there must be $s'\in [x,q] \cup [hq,hx]$
such that $d_Y(s,s')\leq 2\delta$, which contradicts the fact that $q$ and $hq$ are nearest point projections to $\alpha$.
Let $q_1$ (resp. $q_2$) denote the point on $[q,hq]$ that is exactly $4\delta$ apart from $q$ (resp. $hq$).
Let $x_1'$ (resp. $x_2'$) be a nearest point projection to $\gamma$ of the point $q_1$ (resp. $q_2$).
Then $d_Y(x_i',q)\leq 6\delta$ for $i = 1,2$.
By $\delta$-hyperbolicity, if a point $a\in[x_1',x_2']$ is at least $4\delta$ away from both $x_1'$ and $x_2'$, then 
$d_Y(a,[q,hq])\leq 2\delta$.
Let $x_1$ (resp. $x_2$) denote the point on $[x_1',x_2']$ exactly $4\delta$ away from $x_1'$ (resp. $x_2'$).
Put $\gamma_1:=[x,x_1]$, $\gamma_2 := [x_1,x_2]$ and $\gamma_3 := [x_2, hx]$.
Note that $d_Y(x_i,q)\leq 10\delta$ for $i=1,2$, so we put $D_1 := 10\delta$.
By $\delta$-hyperbolicity, except for the $3\delta$ neighborhood of $hq$, points on $[q,hq]$ is in the $\delta$ neighborhood
of $\alpha$.
Hence by putting $D_2:=3\delta$, we have $\gamma_2\subset\mathcal{N}_{D_2}(\alpha)$.
Let $q_1',q_2'$ be nearest point projections of $x_1,x_2$ to $[q,hq]$ respectively.
Then $d_Y(q,q_1')\leq d_Y(q,q_1) + d_Y(q_1,x_1') + d_Y(x_1',x_1) + 2\delta\leq 12\delta$.
By symmetry we have $d_Y(q_2',hq)\leq 12\delta$.
By triangle inequality, we have
we have $\mathrm{length}(\gamma_2)\geq d_Y(q_1',q_2')-d_Y(x_1,q_1')-d_Y(x_2,q_2')\geq d_Y(q,hq)-28\delta$.
Thus we have a required decomposition.
\end{proof}
We are now in a position to prove Theorem \ref{thm.primitive}.
\begin{proof}[Proof of Theorem \ref{thm.primitive}]
Suppose $\omega_n = \phi^k$ for some $\phi\in G$ and $k>1$.
Let $\eta$ be a geodesic axis of $\phi$, and $\gamma = [\omega_0,\omega_n]$.
By Lemma \ref{lem.translen}, $\gamma_2$ of the decomposition of $\gamma = \gamma_1\gamma_2\gamma_3$
from Proposition \ref{prop.fellowtravel} has length at least $Ln$ for some $L>0$ with exponentially high probability.
Let $L':= \mathrm{length}(\gamma_2)$.
Then by applying Lemma \ref{lem.proximal} for small enough $\epsilon$, say $1/100$, 
we may find $D'>0$ such that  
$\mathrm{length}(\gamma_{D'})/\mathrm{length}(\gamma) \geq 1-\epsilon$
with exponentially high probability.
Then we can find a $D'$-proximal point
$q_a\in \gamma_2$ such that $d_\mathcal{C}(q_a, \gamma_2(0))\leq L'\epsilon$.
Let $a$ denote the index that $d_\mathcal{C}(\omega_a,q_a)\leq D'$.
Similarly we can find a point $q_b\in\gamma_2$ such that 
\begin{itemize}
\item $\frac{Ln}{4}\leq d_\mathcal{C}(q_a, q_b)\leq \frac{Ln}{4} + L'\epsilon$,
\item $q_b$ is $D'$-proximal so that $d_\mathcal{C}(\omega_b,q_b)\leq D'$ for $a<b$.
\end{itemize}
We consider translating $[q_a,q_b]\subset \gamma_2$ by $\varphi:=\phi^{\lfloor k/2 \rfloor}$ 
where $\lfloor k/2 \rfloor$ is the largest integer among all integers smaller than $k/2$.
Note that 
$$\frac{\tau(\omega_n)}{3}\leq\tau(\varphi)\leq \frac{\tau(\omega_n)}{2}.$$
By perturbing at most $L'\epsilon$ if necessary, we may assume that 
both $\varphi(\omega_a)$ and $\varphi(\omega_b)$ are within at most $2D_2+2\delta$ distance 
from $D'$-proximal points $q_c, q_d\in\gamma_2$ respectively.
The constant $D_2$ is from Proposition \ref{prop.fellowtravel}.
Hence there exist indices $c,d$ with $a<b<c<d$ such that $d_\mathcal{C}(\omega_i, q_i)\leq D'+2D_2+2\delta$ for $i\in\{a,b,c,d\}$.
Let $\alpha := [\omega_a,\omega_b]$ and $\beta:=[\omega_c,\omega_d]$.
By $\delta$-hyperbolicity, we can decompose $\alpha = \alpha_1\alpha_2\alpha_3$ so that 
$\mathrm{length}(\alpha_1),\mathrm{length}(\alpha_3)<D'+2D_2+4\delta$ and
$\alpha_2\subset \mathcal{N}_{2\delta}(\gamma)$. 
Hence if $n$ is large enough, then for some $t\in[0.1,0.2]$
we can find a $D'$-proximal point
 $q_{a'}\in\gamma_2$ with $d_\mathcal{C}(q_{a'},\alpha(t))\leq 2\delta$.
Similarly, we can also find a $D'$-proximal point $q_{c'}$ such that $d_\mathcal{C}(q_{c'},\beta(t))\leq 2\delta$.
Thus we have indices $a'$ and $c'$ such that $d_\mathcal{C}(\omega_{a'},\alpha(t))\leq D'+2\delta$ and 
$d_\mathcal{C}(\omega_{c'},\beta(t))<D'+2\delta$.
Thus if $\omega_n$ is not primitive we can find indices satisfying conditions (1)-(3) of Lemma \ref{lem.align} for
$M = L/4$ and $D = D' + 2D_2 + 2\delta$.
Therefore the probability that $\omega_n$ is not primitive decays exponentially.
\end{proof}

\section{Random mapping classes are not symmetric}\label{sec.notlift}
The goal in this section is to prove Theorem \ref{thm.notlift}.
We fix a (possibly orbifold) finite covering $\pi:S\map S'$.
A simple closed curve $a\in\mathcal{C}(S)$ is called {\em symmetric} if $\pi(a)$ is also a simple closed curve on $S'$ (see \S \ref{sec.parallel} for more detail).
The first step is to show the exponential decay of the shadow of the set of symmetric curves in $\mathcal{C}(S)$.
To show the exponential decay, we prepare two lemmas (Lemma \ref{lem.notlift} and Lemma \ref{lem.finiteparallel}) 
in \S\ref{sec.zero} and \S\ref{sec.parallel}.
Then \S\ref{sec.exponentialdecay} will be devoted to the proof of the exponential decay.
Finally we prove Theorem \ref{thm.notlift} in \S\ref{sec.proof of 1.3}.

\subsection{Set of symmetric projective measured foliations has $\mu$-stationary measure zero}\label{sec.zero}

Let $\mu$ be a probability measure on the mapping class group $G$ of surface $S$ of finite type.
In this section, we suppose that $\mu$ satisfies Condition \ref{condi.munotlift}.
A measure $\nu$ on $\mathcal{PMF}(S)$ is called $\mu$-stationary if $$\nu(X) = \sum_{g\in G}\mu(g)\nu(g^{-1}X)$$
for any measurable subset $X\subset\mathcal{PMF}(S)$.
We first recall the work of Kaimanovich-Masur.
Recall that a projective measured foliation is said to be {\em uniquely ergodic} if its supporting foliation admits 
only one transverse measure up to scale.
We denote by  $\mathcal{UE}(S)\subset\mathcal{PMF}(S)$ the space of uniquely ergodic foliations with unique 
projective measures.
\begin{thm}[{\cite[Theorem 2.2.4(1)]{KM}}]\label{thm.u-stationary}
There exists a unique $\mu$-stationary probability measure $\nu$ on $\mathcal{PMF}(S)$.
The measure $\nu$ is non-atomic and concentrated on the set of uniquely ergodic foliations $\mathcal{UE}(S)$.
\end{thm}
Similarly as simple closed curves, a projective measured foliation $\lambda$ is said to be {\em symmetric} if $\pi(\lambda)$ is also a projective measured foliation on $S'$.
In this subsection, we will measure by $\nu$ the set of symmetric projective measured foliations.

We now recall the Teichm\"uller space of $S$.
The Teichm\"uller space $\mathcal{T}(S)$ is the space of conformal structures on $S$.
In this paper we consider the Teichm\"uller metric on $\mathcal{T}(S)$;
$$d_\mathcal{T}(\sigma_{1},\sigma_{2})=\frac{1}{2}\log\inf_h K(h),~ (\sigma_{1},\sigma_{2}\in \mathcal{T}(S)),$$
where the infimum is taken over all quasi-conformal maps $h: \sigma_{1}\map \sigma_{2}$
homotopic to the identity, and $K(h)$ is the maximal dilatation of $h$.
Thurston (c.f.\cite{FM}) showed that $\PMF(S)$ compactifies $\mathcal{T}(S)$ so that the action of $G:=\mathrm{Mod}(S)$ extends
continuously.
This compactification is called the {\em Thurston compactification}.
Let $\bar{\mathcal{T}}(S)  := \mathcal{T}(S) \cup \PMF(S)$.

Note that our covering $\pi:S\map S'$ may be an orbifold covering.
If $S'$ is an orbifold, $\mathcal{PMF}(S')$ and $\mathcal{T}(S')$ are 
defined to be the ones
on the surface that we get by puncturing the orbifold points of $S'$.
The covering $\pi$ determines $\Pi:\bT(S')\map \bT(S)$ so that 
$X\in \Pi(\mathcal{T}(S'))$ if $\pi(X)\in\mathcal{T}(S)$, and
$\lambda\in \Pi(\PMF(S'))$ if $\pi(\lambda)\in\PMF(S')$.
As pointed out in \cite[Section 7]{RS}, $\Pi$ is an isometric embedding of $\mathcal{T}(S')$.
We may also extend the $\mu$-stationary measure $\nu$ in Theorem \ref{thm.u-stationary} to $\bT(S)$ by 
$\nu(A) = \nu(A\cap \PMF(S))$ for each subset $A\subset\bT(S)$.
Let $E_{\mathcal{T}}:=\Pi(\bT(S'))$.
Our goal in this subsection is the following lemma.
\begin{lem}\label{lem.notlift}
Let $\mu$ be a probability measure on $G$ which satisfies Condition \ref{condi.munotlift}, and
$\nu$ the $\mu$-stationary measure on $\bT(S)$ from Theorem \ref{thm.u-stationary}.
Then for any finite covering $\pi:S\map S'$, we have for all $g\in G$,
$$\nu(gE_{\mathcal{T}}) = 0.$$
\end{lem}
Recall that $\PMF(S)$ is homeomorphic to the sphere $\mathbb{S}^{6g-7+2n}$.
Although the image $\Pi(\PMF(S'))$ is a sphere of lower dimension, Lemma \ref{lem.notlift} is non-trivial.
This is because the $\mu$-stationary measure $\nu$ is singular to the standard Lebesgue measure on the sphere by the work of Gadre \cite{Gad}.

First, we give a sufficient condition for a subset of $\PMF(S)$ to have $\nu$ measure zero.
\begin{prop}[c.f. {\cite[Lemma 2.2]{KM}}]\label{prop.infinitetranslate}
Let $A$ be a measurable subset of $\mathcal{PMF}(S)$.
Suppose there exist infinitely many distinct translations of $A$ by elements in $\mathrm{gr}(\mu)$.
Suppose further that 
\begin{itemize}
\item[($\ast$)] $\nu(g_1A\cap g_2A) = 0$ or $\nu(g_1A) = \nu(g_2A)$ for all $g_1, g_2\in G$.
\end{itemize}
Then $\nu(A) = 0$.
\end{prop}
\begin{proof}
By ($\ast$), we see that there is some $h\in G$ such that $A':=hA$ satisfies $\nu(A')\geq\nu(gA)$ for all $g\in G$.
Then since $\nu$ is $\mu$-stationary, we have
$$\nu(A') = \sum_{g\in G}\mu(g)\nu(g^{-1}A')\leq \sum_{g\in G}\mu(g)\nu(A') = \nu(A').$$
Thus we see that $\nu(g^{-1}A') = \nu(A')$ for every $g$ in the support of $\mu$.
By discussing the $n$-convolution $\mu^n$ of $\mu$, we see that $\nu(g^{-1}A') = \nu(A')$ for every $g\in\mathrm{sgr}(\mu)$.
Since we have infinitely many distinct translates of $A'$ by elements of $\mathrm{gr}(\mu)$, we see that we also have
infinitely many distinct translates by elements of $\mathrm{sgr}(\mu)^{-1}$. 
Hence we have $\nu(A')=\nu(A)=0$.
\end{proof}

To prove Lemma \ref{lem.notlift}, we recall Teichm\"uller geodesics on the Teichm\"uller space, see for example 
\cite{FM, Gard, KM}
for more details.
Recall that $S$ is a surface of finite type $(g,n)$.
Teichm\"uller showed that for any given point $\sigma\in\mathcal{T}(S)$, 
a holomorphic quadratic differential $q$ determines a geodesic $\Gamma(q)$ with respect to Teichm\"uller metric.
It is also proved that given two points $\sigma_{1},\sigma_{2}\in\mathcal{T}(S)$, there exists a unique Teichm\"uller geodesic $\Gamma(\sigma_{1},\sigma_{2})$ that connects the two.

For $\sigma\in\mathcal{T}(S)$, let $\QD(\sigma)$ denote the Banach space of holomorphic quadratic differentials on $\sigma$ with 
$\parallel \varphi \parallel = \int_\sigma |\varphi|$.
Each $\varphi\in\QD(\sigma)$ determines two measured foliation, called the horizontal foliation and the vertical foliation.
By Riemann-Roch theorem, $\QD(\sigma)$ has complex dimension $3g - 3 + n$.
Let $\mathcal{Q}_0\subset \QD(\sigma)$ denote the unit sphere.
This $\mathcal{Q}_0$ compactifies  $T(\sigma)$ which is called the {\em Teichm\"uller compactification}.

By the work of Hubbard-Masur (compact) and Gardiner (finite type), 
we see:
\begin{lem}[{\cite{HM}, \cite[Chapter 11]{Gard}}]\label{lem.hqd}
For any $\sigma\in\mathcal{T}(S)$ and $F\in\mathcal{PMF}(S)$, 
there is a unique $\varphi\in\QD(\sigma)$ whose horizontal foliation is $F$ up to scale.
\end{lem}

\begin{proof}[Proof of Lemma \ref{lem.notlift}]
The proof goes by induction.
Let $E_{\mathcal{T}}':=gE_{\mathcal{T}}$
and $d$ the complex dimension of $\QD(\sigma')$ for any $\sigma'\in\mathcal{T}(S')$.
We consider intersection $E':=g_1E_{\mathcal{T}}'\cap g_2E_{\mathcal{T}}' \cap \cdots \cap g_nE_{\mathcal{T}}'$.
We first define $d(E')\in\mathbb{N}$.
If $E'\cap\PMF(S)$ contains at most one uniquely ergodic foliation, then we define $d(E')=0$.
In this case we also have $\nu(E') = 0$ since $\nu$ is non-atomic.
If $E'\cap\PMF(S)$ contains at least two uniquely ergodic foliations 
$\mathcal{E}_1, \mathcal{E}_2$, then there is a unique Teichm\"uller geodesic $\gamma$ connecting 
$\mathcal{E}_1$ and $\mathcal{E}_2$ by \cite{GM}.
Since covering maps induce isometric embeddings of Teichm\"uller spaces \cite[Section 7]{RS}, 
any point of $\gamma$ is in $E'$.
In particular $E'\cap\mathcal{T}(S)$ is non-empty.
For any $\sigma\in E'\cap\mathcal{T}(S)$, each $g_iE_{\mathcal{T}}'$ determines a subspace of $S_i(\sigma)\subset\QD(\sigma)$ 
which consists of the lifts of holomorphic quadratic differentials with respect to the covering $\pi\circ g_i^{-1}$.
Let $S(\sigma):=\cap_{i=1,\dots,n}S_i(\sigma)$ and $d(\sigma):= \dim S(\sigma)$.
Since $d(\sigma)\in\mathbb{N}$, there exists $\sigma'\in E'\cap\mathcal{T}(S)$ such that $d(\sigma')\geq d(\sigma)$ for any $\sigma\in E'\cap\mathcal{T}(S)$.
We define $d(E'):=d(\sigma')$.
Then we explain how the induction works by using a style of inductive algorithm, see Algorithm \ref{algo.mvit} which is named M$\nu$IT.
By $\mathrm{M}\nu\mathrm{IT}(E, d)$, we have $\nu(E) = 0$.
Note that although the depth of Algorithm \ref{algo.mvit} is finite, the width is infinite.

\renewcommand{\algorithmicrequire}{\textbf{Input:}}
\begin{algorithm}
\caption{M$\nu$IT(\underline Measure by $\underline\nu$ the \underline Intersection of \underline Translates)}
\label{algo.mvit}
\begin{algorithmic}                  
\REQUIRE 
$(E':=g_1E_{\mathcal{T}}'\cap g_2E_{\mathcal{T}}' \cap \cdots \cap g_nE_{\mathcal{T}}', d(E'))$.
\ENSURE $\nu(E') = 0$. 
	\IF{$d(E') = 0$}
		\STATE By the definition of $d(E')$ and Lemma \ref{lem.hqd}, we have $\nu(E') = 0$.
	\ENDIF
	\FOR{$h_1, h_2\in G$}
		\STATE 
		Let $E'_1 := h_1E'$ and $E'_2 := h_2E'$. 
		Note that since each $g\in G$ induces a vector isomorphism between $\QD(\sigma)$ and $\QD(g\sigma)$,
		we have $d(E'_1) = d(E'_2) = d(E')$.
				 \IF{$d(E'_1\cap E'_2) = d(E')$}
				 	\STATE We see that $\nu(E'_1) = \nu(E'_2)$ by 
					Lemma \ref{lem.hqd}.
				\ELSE
					\STATE 
					In this case we have $d(E')> d(E'_1\cap E'_2)$.
					Then we apply \\M$\nu$IT$(E'_1\cap E'_2, d(E'_1\cap E'_2))$,
					which proves $\nu(E'_1\cap E'_2) = 0$.
					
				 \ENDIF
	\ENDFOR
\STATE 
We have seen that the condition $(\ast)$ of Proposition \ref{prop.infinitetranslate} is satisfied.
By Condition \ref{condi.munotlift} and the north-south dynamics of pseudo-Anosov maps (see \cite{Thu2}),
we see that there are infinitely many translates of $E'$.
Thus by Proposition \ref{prop.infinitetranslate}, we have $\nu(E') = 0$.
\end{algorithmic}
\end{algorithm}

\end{proof}

\subsection{Upper bound for the number of parallel translates of the set of symmetric curves}\label{sec.parallel}
Recall that we have fixed a (possibly orbifold) covering $\pi:S\map S'$.
If $S'$ is an orbifold, we define $\mathcal{C}(S')$ as the curve graph of the surface that we get by puncturing every 
orbifold point of $S'$.
We define one to finite relation 
$\Pi_{\mathcal{C}}:\mathcal{C}(S')\map \mathcal{C}(S)$ as follows.
A curve $b\in \mathcal{C}(S)$ is in $\Pi_{\mathcal{C}}(a)$ for some $a\in\mathcal{C}(S')$ if 
$\pi(a) = b$ as isotopy classes of simple closed curves.
In \cite{RS}, Rafi-Schleimer showed that $\Pi_{\mathcal{C}}$ is quasi-isometric (Theorem \ref{thm.RS}).
Hence the map $\Pi_{\mathcal{C}}$ extends continuously to the Gromov boundary $\partial\mathcal{C}(S')$.
Let $E$ denote $\Pi_\mathcal{C}(\mathcal{C}(S')\cup\partial\mathcal{C}(S'))$.
We call elements in $E$ symmetric.
We consider translates $gE$'s  of $E$ by $g\in G$.
Our aim in this subsection is to prove the following lemma.

\begin{lem}\label{lem.finiteparallel}
For any $D_0>0$, there exist $D_1, D_2>0$ which depends only on $S$ and $D_0$ such that for any $a,b\in \mathcal{C}(S)$ with $d_\mathcal{C}(a,b)>D_1$, the number of elements in
$$\mathcal{P}(a,b,D_{0}):=\{gE\mid d_\mathcal{C}(a,gE)<D_0 \text{ and } d_\mathcal{C}(b,gE)<D_0\}$$
is bounded from above by $D_{2}$.
Here we count the number of images i.e. if $g_{1}E = g_{2}E$ as subsets, we just count one time.
\end{lem}

For a proof, we need the notion of subsurface projection.
A subsurface $Y\subset S$ is called {\em essential} if each component of $\partial Y$ is an essential simple closed curve.
Unless otherwise stated, we always assume that subsurfaces are essential.
Given a subsurface $Y\subset S$ which is not an annulus nor three holed sphere, we define subsurface projection $\pi_Y:\mathcal{C}(S)\map \mathcal{C}(Y)$ as follows:
given a curve $a\in \mathcal{C}(S)$ on $S$, arrange $a$ so that it has minimal intersection with $Y$.
Take a component $a'$ of $Y\cap a$ and consider a small neighborhood $N$ of $a'\cup \partial Y$.
Then $\pi_{Y}(a)$ is defined to be a component of $N$ which is in $\mathcal{C}(Y)$.
If $a$ does not intersect with $Y$, then we define $\pi_{Y}(a) = \emptyset$.
If $Y$ is an annulus, we need special care, however we do not need the detail for the proof, so we omit the definition.
See for example \cite{MM2, Rafi} for the detail.
If $Y$ is a three holed sphere, subsurface projection is not defined.
We call a subsurface $Y\subset S$ {\em symmetric} if it is a component of $p^{-1}(Y')$ for some $Y'\subset S'$.
Given two curves $a,b\in\mathcal{C}(S)$, we let $d_{Y}(a,b):= \mathrm{diam}(\pi_{Y}(a),\pi_{Y}(b))$.
If $Y$ is an annulus with core curve $\alpha$, we often use $d_{\alpha}$ to denote $d_{Y}$.
Rafi-Schleimer showed the following lemma.
\begin{lem}[{\cite[Lemma 7.2]{RS}]}]\label{lem.RS}
There exists $T_{1}$ which depends only on $S$ and the degree of $\pi:S\map S'$ such that for any subsurface $Y\subset S$ and $a,b\in E$, 
if $d_{Y}(a,b)\geq T_{1}$ then $Y$ is symmetric.
\end{lem}

We recall the work of Masur-Minsky \cite{MM2}.
\begin{thm}[{\cite[Theorem 3.1]{MM2}}, Bounded geodesic image]\label{thm.MM2}
There exists a constant $M_{1}>0$ which depends only on $S$ with the following property.
Let $Y\subset S$ be a proper subsurface $Y$ which is not a three holed sphere.
Let $\gamma$ be a geodesic in $\mathcal{C}(S)$ with $\pi_Y(v)\not=\emptyset$ for all vertex $v$ on $\gamma$.
Then $$\mathrm{diam}_Y(\gamma)\leq M_{1},$$
where $\mathrm{diam}_Y(\gamma)$ is the diameter of $\pi_Y(\gamma)$ in $\mathcal{C}(Y)$.
\end{thm}

Combining Lemma \ref{lem.RS} and Theorem \ref{thm.MM2}, we have the following.
\begin{lem}\label{lem.intersection}
Given $g\in G$, fix $a,b\in gE$ and $D>0$.
Suppose there exists a subsurface $Y\subset S$ such that $d_Y(a,b)\geq T+2M_{1}$ for $T\geq T_{1}$, and 
$d_\mathcal{C}(\{a,b\},\partial Y)\geq D+2$.
Then if there are $c,d\in hE$ for some $h\in G$ such that 
$d_\mathcal{C}(c,a)\leq D$ and $d_\mathcal{C}(d,b)\leq D$
then we have $d_Y(c,d)\geq T$.
In particular $Y$ is symmetric for both $g\pi g^{-1}$ and $h\pi h^{-1}$.
\end{lem}
\begin{proof}
Since we assume that $\partial Y$ is far from $a,b,c,d$, we see that every vertex on geodesics $[a,c]$ and $[b,d]$ intersects $Y$ non-trivially.
Hence by Theorem \ref{thm.MM2}, we see that $d_Y(c,d)\geq d_Y(a,b) - d_Y(c,a) - d_Y(b,d) \geq T_1$.
The last assertion follows from Lemma \ref{lem.RS}.
\end{proof}


We recall the following work of Rafi-Schleimer \cite{RS} and Rafi \cite{Rafi}.
\begin{thm}[\cite{RS}]\label{thm.RS}
The covering relation $\Pi:\mathcal{C}(S)\map\mathcal{C}(S')$ is a $Q$-quasi-isometric embedding.
The constant $Q$ depends only on $S$ and the degree of $\pi:S\rightarrow S'$.
\end{thm}
To state the work of Rafi \cite{Rafi}, we need the notion of {\em shortest markings}.
Given a point $\sigma\in\mathcal{T}(S)$, the shortest marking of $\sigma$ is the set of curves chosen as follows.
Consider $\sigma$ as a hyperbolic structure of $S$. 
First we greedy choose the shortest curves; let $\alpha_{1}$ be the shortest curve, then we choose $\alpha_{2}$ as the shortest curve on $S\setminus \alpha_{1}$. 
We proceed until $\alpha_{i}'s$ give a pants decomposition of $S$.
Then we choose $\beta_{i}$ as the shortest curve among curves intersecting only $\alpha_{i}$.
We denote the shortest marking of $\sigma$ by $\mu({\sigma})$.
In the statement of Theorem \ref{thm.Rafi}, the function $[x]_{k}$ is equal to zero when $x < k$ and is equal to $x$ when $x \geq k$.
We also modify the $\log$ in the statement so that $\log x=0$ for $x\in[0,1]$.
\begin{thm}[\cite{Rafi}]\label{thm.Rafi}
There exists $k'>0$ such that for $k>k'$, for any $\sigma_{1},\sigma_{2}\in\mathcal{T}(S)$, the following holds.
Let $A := d_\mathcal{T}(\sigma_1,\sigma_2)$, and
$$B := \sum_Y[d_Y(\mu(\sigma_{1}), \mu(\sigma_{2}))]_{k}  + \sum_\alpha\log[d_\alpha(\mu(\sigma_1), \mu(\sigma_2))]_{k}.$$
Where in the first sum, $Y$ is taken over all subsurfaces of $S$ which are not three holed spheres nor annuli, and in the second sum $\alpha$ is taken over all essential simple closed curves.
Then there exist constants $C,c>0$ which depend only on $S$ and $k$ such that
$$\frac{1}{C}A-c \leq B \leq CA+c.$$
\end{thm}

We are now ready to prove Lemma \ref{lem.finiteparallel}.
\begin{proof}[Proof of Lemma \ref{lem.finiteparallel}]
Throughout the proof, we call a constant {\em universal} if it only depends on $S$ and the degree of $\pi:S\rightarrow S'$.
For a given $s\in \mathcal{C}(S)$, let $\sigma(s,gE)$ denote a point $\mathcal{T}(S)$ so that the shortest marking $\mu(s,gE):=\mu(\sigma(s,gE))$
contains a closest point projection of $s$ to $gE$.
Note that there may be several closest projections. 
We choose one of them, and fix it.
By considering the conjugacy of the covering, we may suppose $E\in\mathcal{P}(a,b,D_{0})$.
%

We proceed by induction. First, note that once we bound the degree of the covering, the statement of Lemma  \ref{lem.finiteparallel} is true for annuli.
Then we may suppose Lemma \ref{lem.finiteparallel} holds for any subsurface of $S$.
Let $D_{3}'$ be the constant so that Lemma \ref{lem.finiteparallel} holds for 
$D_{0} := M_{1}$ and $D_1 := D_3'$ for any subsurface of $S$.
Then let $D_{3}:=\max\{D_{3}', T_{1}\}$.

We now consider two cases.
The first case is where we can find four subsurfaces $Y_{i} (i = 1,2,3,4)$ such that 
\begin{enumerate}
	\item $d_\mathcal{C}(a,\partial Y_{i})>2D_{0}+2, d_\mathcal{C}(b,\partial Y_{i})>2D_{0}+2$,
	\item $d_\mathcal{C}(\partial Y_{i},\partial Y_{j})>3$,
	\item $d_{Y_{i}}(\mu(a,E),\mu(b,E))\geq  D_{3} + 6M_{1}$.
\end{enumerate}
Note that by the third condition and Theorem \ref{thm.MM2}, any geodesic in $\mathcal{C}(S)$ connecting $\mu(a,E)$ and $\mu(b,E)$
passes close to $\partial Y_{i}$'s, and we also suppose $\partial Y_{i}$'s appear in the order of the index.
By Lemma \ref{lem.intersection} we see that for any $gE\in \mathcal{P}(a,b,D_{0})$, 
we have $d_{Y_{i}}(\mu(a,gE),\mu(b,gE))\geq D_{3} + 4M_{1}$ and hence
$\partial{Y_{i}}$ are all contained in every $gE\in \mathcal{P}(a,b,D_{0})$.
Then again by Theorem \ref{thm.MM2}, we may choose components $y_{1}$ and $y_{4}$ of $\partial Y_{1}$ and $\partial Y_{4}$ respectively so that
$d_{Y_{2}}(y_{1}, y_{4})\geq D_{3}$.
Hence for any  $gE\in \mathcal{P}(a,b,D_{0})$, we have $d_{Y_{2}}(\mu(y_{1},gE),\mu(y_{4},gE))\geq D_{3}$.
By induction, for $gE\in \mathcal{P}(a,b,D_{0})$, the restriction $\pi\circ g^{-1}|Y_2$ of the covering is one of the universally bounded number of coverings from $Y_2$.
Moreover, once we fix the topology of a subsurface $Y'$,
there are only finitely many possible embedding $Y'\hookrightarrow S'$ up to the action of mapping classes on $S'$ that can be lifted to $S$ via $\pi:S\map S'$.
Let $\mathcal{P}(a,b,D_0,Y_{2})$ denote the subset of $\mathcal{P}(a,b,D_{0})$ whose elements correspond $g$'s that satisfy the following;  
$\pi|Y_{2}$ and $\pi\circ g^{-1}|Y_{2}$ are the same as coverings from $Y_{2}$,
and the image of $Y_{2}$ in $S'$ by the coverings are related by a mapping class $\phi_{g}$ which can be lifted to $S$.
We see that $\mathcal{P}(a,b,D_{0})$ can be decomposed into universally bounded number of subsets of type $\mathcal{P}(a,b,D_0,Y_{2})$.
To find a universal bound for the cardinality of $\mathcal{P}(a,b,D_{0})$, we only need to find a universal bound for the cardinality of each subset in the decomposition.
Hence, it suffices to find a universal bound for the number of elements in $\mathcal{P}(a,b,D_0,Y_{2})$.

Let $gE\in\mathcal{P}(a,b,D_0,Y_{2})$.
Note that if $\widetilde{\varphi}:S\map S$ is a lift with respect to $\pi$, then $\widetilde{\varphi}E = E$.
Hence by precomposing suitable lift $\widetilde{\phi_{g}}$ of $\phi_{g}$ to $g$, which we again denote by $g$ by abuse of notation, we may suppose 
that $\pi(g^{-1}(Y_{2})) = \pi(Y_{2})$ and $\pi(y_{j})\cap \pi(Y_{2}) = \pi g^{-1}(y_{j})\cap \pi(Y_2)$ for both $j = 1,4$.
Let $y_{2}'$ be a component of $\partial\pi(Y_{2})$.
Then by the above observation, 
we have for any component $z_2$ of $\pi^{-1}(y_{2}')$,
\begin{eqnarray}\label{eq.balance}
\frac{i(y_{2}',\pi(y_{1}))}{i(y_{2}', \pi(y_{4}))} = \frac{i(z_2, \pi^{-1}\pi(y_{1}))}{i(z_2, \pi^{-1}\pi(y_{4}))} = 
\frac{i(g(z_2), \pi^{-1}\pi(y_{1}))}{i(g(z_2), \pi^{-1}\pi(y_{4}))}.
 \end{eqnarray}
Note that since $\pi^{-1}(\pi(y_{1}))$ and $\pi^{-1}(\pi(y_{4}))$ fill the surface $S$, 
it determines a quadratic differential $q$ with horizontal foliation $\pi^{-1}(\pi(y_{1}))$ and vertical foliation $\pi^{-1}(\pi(y_{4}))$.
Given a Teichm\"uller geodesic $\Gamma:\mathbb{R}\rightarrow \mathcal{T}(S)$, 
a simple closed curve is said to be {\em balanced} at time $t$ on $\Gamma$ if the intersection number with horizontal foliation and vertical foliation of quadratic differential determined by $\Gamma$ and $t$ coincide.
By (\ref{eq.balance}), on the Teichm\"uller geodesic $\Gamma(q)$ determined by $q$, 
all component of $\pi^{-1}(y_{2}')\cup g\pi^{-1}(y_{2}')$ are balanced at the same time.
By Theorem \cite[Theorem 6.1]{Rafishort} applied to $\Gamma(q)$,
we see that $\pi^{-1}(y_{2}')\cup g\pi^{-1}(y_{2}')$ is a disjoint union of simple closed curves.
Since the number of disjoint simple closed curves is universally bounded from above, we may further decompose
$\mathcal{P}(a,b,D_0,Y_{2})$ into universally bounded number of subsets so that if $g_{1}E$ and $g_{2}E$ are in the same subset,
we have $g_{1}\pi^{-1}(y_{2}') = g_{2}\pi^{-1}(y_{2}')$.
We argue similarly for $Y_{3}$ and hence all we need is to find a universal bound for the number of $\{g_{i}E\}_{i\in I}$ with 
\begin{itemize}
	\item $g_{i}\pi^{-1}(y_{2}') = g_{j}\pi^{-1}(y_{2}')$, and
	\item $g_{i}\pi^{-1}(y_{3}') = g_{j}\pi^{-1}(y_{3}')$ for some component $y_{3}'$ of $\partial \pi (Y_{3})$
\end{itemize}
for any $i,j\in I$.
In this case, the Teichm\"uller geodesic determined by the quadratic differential 
with horizontal foliation $g_{i}\pi^{-1}(y_{2}')$ and vertical foliation $g_{i}\pi^{-1}(y_{3}')$
 is contained in $g_iE_\mathcal{T}$ for all $i\in I$.
Then by Lemma \ref{lem.symmetricqd}, we have a desired bound.
Thus we are done for the first case.



We now consider the case where we can not find subsurfaces satisfying the conditions of the first case.
The shadow of a Teichm\"uller geodesic $\Gamma$ is the set of all curves which are the shortest at some point on $\Gamma$.
In \cite{MM}, it is proved that the shadow of any Teichm\"uller geodesic is a (unparametrized) quasi-geodesic.
Hence there exists a universal constant $D_{4}$ such that for $hE\in\mathcal{P}(a,b,D_{0})$,
the shadow of the Teichm\"uller geodesic $\Gamma(\sigma(a,hE), \sigma(b,hE))$ is contained in 
$D_{4}$-neighborhood of the shadow of $\Gamma(\sigma(a,E), \sigma(b,E))$.
Then by assuming $D_{1}$ in the statement large enough, we may suppose that 
there are points $\sigma_1,\sigma_2$ on the Teichm\"uller geodesic $\Gamma(\sigma(a,E),\sigma(b,E))$ such that
\begin{itemize}
	\item $d_\mathcal{C}(\mu(\sigma_1), \mu(\sigma_2)) > D_{4} +2$, and
	\item $d_Y(\mu(\sigma_1),\mu(\sigma_2)) < D_3 + 6M_1$ for all proper subsurface $Y\subset S$,
	\item $d_Y(\mu(\mu(\sigma_1), hE),\mu(\mu(\sigma_2),hE)) < D_3 + 6M_1$ for all proper subsurface $Y\subset S$.
\end{itemize}
Let $D_{5}:=D_3 + 6M_1$ and $\sigma_1' := \sigma(\mu(\sigma_1),hE)$.
These condition together with Theorem \ref{thm.MM2} imply that
$d_Y(\mu(\sigma_1), \mu(\sigma_1'), hE))<2D_{5} +M_1$ for any proper subsurface $Y\subset S$.
Hence by Theorem \ref{thm.Rafi}, we see that there exists a universal constant $D_{6}$ and $\epsilon$ such that 
$d_\mathcal{T}(\sigma_1,\sigma_1')<D_{6}$ and the Teichm\"uller geodesic connecting $\sigma_1$ and $\sigma_1'$ is contained in the $\epsilon$-thick part of $\mathcal{T}(S)$.

Recall that the subgroup of $\mathrm{Mod}(S')$ that can be lifted via $\pi:S\map S'$ is of finite index, and the thick part of the moduli space of $S'$ is compact.
Hence we can find $h'\in G$ such that $h'\sigma\in hE_{\mathcal{T}}$ and we can connect $\sigma_1'$ and $h'\sigma_1$ by a Teichm\"uller geodesic which is contained in the $\epsilon$-thick part and of length universally bounded from above.
Thus we have a universal constant $D_{7}$ such that $\sigma_1$ and $h'\sigma_1$ can be connected by a path in the $d\epsilon$-thick part of length less than $D_{7}$.
Again, by the compactness of the thick part of the moduli space of $S$, 
the number of such mapping classes are universally bounded.
Hence we see that if $hE\in\mathcal{P}(a,b,D_{0})$, $hE$ has to pass through one of the universally bounded number of points.
By assuming $D_1$ large enough we may apply the same argument to different subarc of $\Gamma(\sigma(a,E),\sigma(b,E))$, 
and hence we see that there are two disjoint finite subsets of $\mathcal{T}(S)$ so that if $hE\in\mathcal{P}(a,b,D_{0})$, $hE$ has to pass through both subsets.
Since two points in Teichm\"uller space determines a quadratic differential, again by Lemma \ref{lem.symmetricqd}, we have a desired bound.
This completes the proof of Lemma \ref{lem.finiteparallel}.
\end{proof}
Finally we prove the following lemma which we used in the proof of Lemma \ref{lem.finiteparallel} twice.
\begin{lem}\label{lem.symmetricqd}
There exists $K>0$ which depends only on $S$ such that
for any $q\in\QD(S)$, $\sharp\{gE_\mathcal{T}\mid\Gamma(q)\subset gE_\mathcal{T}\}<K$.
\end{lem}
\begin{proof}
By taking conjugation if necessary, we may suppose $E_\mathcal{T}\in\{gE_\mathcal{T}\mid\Gamma(q)\subset gE_\mathcal{T}\}$.
Recall that by integrating the square root, each non-zero element $q\in\QD(S)$ determines a singular Euclidean structure with horizontal and vertical foliation.
Let $\Sing(q)$ denote the set of singular points of the singular Euclidean structure.
This $\Sing(q)$ is finite.
Pick any $s\in\Sing(q)$, then we define $\Sigma_1(s):=\pi^{-1}\pi(s)$ where $\pi:S\map S'$ is the finite covering we fixed above.
Inductively define $\Sigma_{i+1}(s):= g\circ\pi^{-1}(\pi\circ g^{-1}(\Sigma_i(s)))$.
Since $\Sigma_i(s)\subset \Sigma_{i+1}(s)\subset \Sing(q)$, we eventually have $\Sigma_i(s) = \Sigma_{i+1}(s) (=:\Sigma(s))$ for large enough $i$.
Next, we pick any $x\in S\setminus \Sing(q)$.
There is a point $s'\in\Sing(q)$ such that we can connect $x$ and $s'$ by a single Euclidean geodesic $\gamma$.
The geodesic $\gamma$ has well defined angle $\theta_\gamma\mod \pi$.
Let $l_q(\gamma)$ denote the Euclidean length of $\gamma$.
Since there are only finitely many points from $\Sigma(s')$ with angle $\theta_\gamma$ and Euclidean distance $l_q(\gamma)$,
we get $\Sigma(x)\subset S\setminus\Sing(q)$ in the same way as above.
Thus we get an equivalence relation $x\sim y:\iff y\in\Sigma(x)$ on $S$. 
Since this relation is defined by composing local homeomorphisms $g$ and $\pi$, the quotient map $\pi':S\map S/\!\!\sim$ is a covering.
By construction, $\pi'$ factors through $\pi:S\map S'$ and we have two coverings $p,p_{g}:S'\rightarrow S/\!\!\sim$ 
such that
$p\circ\pi = p_{g}\circ\pi\circ g$ as covering maps.
Furthermore, since for each $x\in S/\!\!\sim$, we may find a small open neighborhood $U_{x}$ such that
on all component of $(\pi')^{-1}(U_{x})$, we can identify the quadratic differentials via $\pi$ and $\pi g^{-1}$,
we have a quadratic differential $q'$ with $(\pi')^{-1}(q') = q$.
Let us suppose that for $p,p_{g}:S'\rightarrow S/\!\!\sim$, 
we have $p_{*}(\pi_{1}(S')) = (p_{g})_{*}(\pi_{1}(S'))$.
Then there exists a homeomorphism $f_{g}:S'\rightarrow S'$ such that $p_{g} = p\circ f_{g}$.
We further suppose that $(f_{g}\circ\pi)_{*}(\pi_{1}(S)) = \pi_{*}\pi_{1}(S)$ in $\pi_{1}(S')$.
Then we can lift $f_{g}$ to $\widetilde{f_{g}}:S\rightarrow S$.
By construction, $(\widetilde{f_{g}})^{-1}g$ preserves $q$.
Suppose there is $hE_\mathcal{T}\not=gE_\mathcal{T}\in\{gE_\mathcal{T}\mid\Gamma(q)\subset gE_\mathcal{T}\}$
such that $hE_\mathcal{T}$ determines the same equivalence relation $\sim$ and  similarly as $g$, we have a map
$\widetilde{f_{h}}$ which is a lift of a homeomorphism $f_{h}:S'\rightarrow S'$ and $(\widetilde{f_{h}})^{-1}h$ preserves $q$.
If $(\widetilde{f_{g}})^{-1}g = (\widetilde{f_{h}})^{-1}h$, we have $h^{-1}g =  (\widetilde{f_{h}})^{-1}\widetilde{f_{g}}$ and hence
$hE_\mathcal{T}=gE_\mathcal{T}$.
Hence $(\widetilde{f_{g}})^{-1}g \not= (\widetilde{f_{h}})^{-1}h$.
The number of mapping classes that preserve $q$ is universally bounded.
Furthermore the number of possibility of $S/\!\!\sim$ is universally bounded and for each case
the number of coverings $S'\rightarrow S/\!\!\sim$ is also universally bounded in terms of $S$.
Since the number of subgroups of a fixed degree in $\pi_{1}(S')$ is also bounded,
there is an upper bound which depends only on $S$ for the number of $\{gE_\mathcal{T}\mid\Gamma(q)\subset gE_\mathcal{T}\}$.
The number of possible coverings from $S$ is finite and thus we complete the proof.
\end{proof}
\subsection{Exponential decay for the shadow of $E$}\label{sec.exponentialdecay}
By the work of Klarreich \cite{Kla} (see also Hamenst\"adt \cite{Ham}), the Gromov boundary $\partial\mathcal{C}(S)$ of 
$\mathcal{C}(S)$ is identified with the space $\mathcal{F}_\mathrm{min}(S)$ of minimal foliations.
There is a natural measure forgetting map from $\mathcal{UE}(S)$ to $\mathcal{F}_\mathrm{min}(S)$.
Hence we may consider the push forward of $\nu$ to $\mathcal{F}_\mathrm{min}(S)$,
which we again write as $\nu$ by abuse of notation.
This $\nu$ extends to $\bC(S):=\mathcal{C}(S)\cup\partial\mathcal{C}(S)$ by $\nu(A) = \nu(A\cap\partial\mathcal{C}(S))$
for $A\subset \bC(S)$.


For a subset $A\subset \bC(S)$, we define the shadow $S_p(A,r)$ for $r>0$ and $p\in \bC(S)$ by
$$S_p(A,r) := \bigcup_{a\in A}S_p(a,r).$$
We first prove the following lemma, which is a key step for showing Theorem \ref{thm.notlift}.
\begin{lem}[c.f. {\cite[Lemma 2.10]{Mah2012}}]\label{lem.setshadow}
There is a constants $K>0$ and $c < 1$, such that 
for any $r>0$ and $g\in G$,
$$\nu(S_1(gE,r)) < c^r, \mathbb{P}(\omega_n\in(S_1(gE,r))) < Kc^r, $$
and the constants 
$K$ and
$c$ depend on $\mu$ and $\pi:S\rightarrow S'$ but not on $r,g$ and $n$.
\end{lem}
We prove Lemma \ref{lem.setshadow} by borrowing several arguments from the proof of \cite[Lemma 2.10]{Mah2012}.
In \cite{Mah2012}, Maher uses several lemmas from \cite{CM}, which are applications of Lemma \ref{lem.fourpt} below.
Instead of using those lemmas, we only use Lemma \ref{lem.fourpt} since the proof of each lemma in \cite{CM} that we need 
is short and elementary.
\begin{lem}[{see for example \cite[Proposition 6.7]{Bow2}}]\label{lem.fourpt}
Let $(X,d_X)$ be a $\delta$-hyperbolic space.
Then there is a constant $K_1$ which depends only on $\delta$ with the following property.
For any four points $x_1,x_2,x_3,x_4\in X$, 
there is an embedded tree $T$ connecting the four point such that 
\begin{subequations}
\begin{eqnarray}
d_T(x_i,x_j)\leq d_X(x,y) + K_1\label {eq.t-dis}\\
(x_i\cdot x_j)_{x_k} - 2K_1 \leq (x_i \cdot x_j)^T_{x_k} \leq (x_i \cdot x_j)_{x_k} + K_1\label{eq.t-gromov}
\end{eqnarray}
\end{subequations}
for $1\leq i,j\leq 4$.
Where $d_T$ denotes the distance in $T$, and for $a,b,c\in T$, $(a,b)_c^T$ denotes the Gromov product with respect to $d_T$.
\end{lem}
Note that the only combinatorial type of the tree up to reindexing is as depicted in Figure \ref{fig.tree}.
\begin{figure}[htp]
\begin{center}
\includegraphics[scale = 0.25]{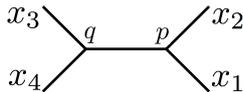}
\caption{Approximate tree.}
\label{fig.tree}
\end{center}
\end{figure}

We will use the following lemma in \cite{Mah2012}.
\begin{lem}[{\cite[Proposition 2.12.]{Mah2012}}]\label{lem.pointshadowdecays}
For any $\epsilon > 0$, there is a constant $K_2(\epsilon)$ which depends on $\epsilon$ and $\mu$,
such that if $r \geq K_2(\epsilon)$, then $\nu(S_1(x, r)) < \epsilon$.
\end{lem}
For the proof of Lemma \ref{lem.setshadow}, we also prepare the following lemma.
\begin{lem}\label{lem.setshadowdecays}
There exists $\epsilon > 0$ and a constant $K_3$ which depend only on $\mu$ and $\pi:S\rightarrow S'$,
such that if $r \geq K_3$, then $\nu(S_1(gE, r)) < 1-\epsilon$ for any $g\in G$.
\end{lem}
\begin{proof}
First we prove that for a given $g\in G$ and $\epsilon>0$, there is some number $r_{g,\epsilon}$ such that $\nu(S_1(gE,r_{g,\epsilon}))<\epsilon$.
The argument below is almost the same as \cite[Proposition 2.12]{Mah2012}.
Suppose contrary that there exists a sequence $\{r_i\}_{i\in I}$ with $r_i\rightarrow\infty$ and $\nu(S_1(gE,r_i))\geq\epsilon$.
Then since $S_1(gE,r_i)\subset S_1(gE,r_j)$ if $j>i$, we have $U:= \cap_{i\in I} S_1(gE,r_i)\geq\epsilon$.
But if $\lambda\in U$, there is a sequence $\{x_i\}$ of points in $gE$ such that $(x_i\cdot\lambda)_1\rightarrow\infty$.
Since $gE$ is a quasi-isometrically embedded image of $\mathcal{C}(S')$, we have $\lambda\in gE$.
By Lemma \ref{lem.notlift}, we have a contradiction.

Then recall that $S_1(a,r)$ can also be regarded as the set of points with geodesics connecting $1$ and those points pass through certain neighborhood of $a$.
This, together with the hyperbolicity of $\mathcal{C}(S)$, and the fact $gE$ is a quasi-isometrically embedded image of $\mathcal{C}(S')$, implies 
for sufficiently large $r,K$,
if $S_1(gE,r)$  intersects both $S_1(a_1,K)$ and $S_1(a_2,K)$ with $a_1$ and  $a_2$ sufficiently far apart, 
then $gE\in\mathcal{P}(a_1,a_2, D_{0})$ for some constant $D_{0}$ depending only on $S$ but not on $d_\mathcal{C}(a_1,a_2)$.
Let us take $a_1$ and $a_2$ so that $d_\mathcal{C}(a_1,a_2)$ is sufficiently large so that we can apply Lemma \ref{lem.finiteparallel}.
Furthermore, we may choose $a_1,a_2$ from the semigroup generated by the support of $\mu$ so that we have $\nu(S_1(a_i,K))>0$ for $i=1,2$.
Let $\epsilon:=\min(\nu(S_1(a_1,K), \nu(S_1(a_2,K))$.
Then suppose $S_1(gE, r)$ is disjoint from $S_1(a_i,K)$ for one of $i=1,2$, then $\nu(S_1(gE, r)) < 1-\epsilon$.
On the other hand, by Lemma \ref{lem.finiteparallel}, there are only bounded number of many $g_{j}E$ such that $S_1(g_jE, r)$ intersects 
$S_1(a_i,K)$ for both $i=1,2$.
For each $j$ there exists $r_j$ such that $\nu(S_1(g_jE, r_{j}))<1-\epsilon$.
Then for  $K_{3}:=\max(r_j,r)$, we have $\nu(S_1(gE, r)) < 1-\epsilon$ for any $r>K_{3}$ and $g\in G$.
\end{proof}
We recall the following lemma which is a version of the lemma due to Maher.
For later convenience, we slightly modify the constants in (4) and (5), but the proof goes exactly the same way as \cite{Mah2012}.
\setcounter{equation}{0}
\begin{lem}[{\cite[Lemma 2.11.]{Mah2012}
}]\label{lem.Mahkey}
Let $\mu$ be a probability distribution of finite support of diameter $D$.
Let $X_0 \supset X_1 \supset X_2 \supset \dots$ be a sequence of nested closed subsets of 
$\bC(S)$ with the following properties:
\begin{eqnarray}
1 &\not\in & X_0 \\
(\mathcal{C}(S) \setminus X_n) \cap X_{n+1} &=& \emptyset\\
d_\mathcal{C}(\mathcal{C}(S) \setminus X_n, X_{n+1}) &\geq & D
\end{eqnarray}
Furthermore, suppose there is a constant 
$0 < \epsilon < 1$ such that, for any $x \in X_n \setminus X_{n+1}$ which is the translate of the base point $p$ by $x\in G$,  
\begin{eqnarray}
\nu_x(X_{n+2}) \leq 1-\epsilon\\
\nu_x(\mathcal{C}(S) \setminus X_{n-1})\leq\epsilon/2
\end{eqnarray}
where $\nu_x(A):= \nu(x^{-1}A)$ for any $A\subset\bC(S)$.
Then there are constants $c < 1$ and $K$, which depend only on $\epsilon$ and $\mu$, such that 
$\nu(X_n) < c^n$ and $\mathbb{P}(\omega_i\in X_n) < Kc^n$ for all $i\in \mathbb{N}$.
\end{lem}

Then, to prove Lemma \ref{lem.setshadow}, it suffices to prove
\begin{lem}\label{lem.key}
There exists $L $ which depends on $\mu,\delta$ with the following property.
The sets $X_n:=S_1(gE,L(n+1))$ for all $n\in\mathbb{N}$ form a sequence of nested sets which satisfies
(1)-(5) in Lemma \ref{lem.Mahkey}.
\end{lem}
\begin{proof}
The proof goes in a similar way to \cite[Lemma 2.13]{Mah2012}.
Let $D$ be the diameter of $\mu$.
We use the constants $K_1, \dots, K_3$ from Lemma \ref{lem.fourpt}-\ref{lem.setshadowdecays}.
Let $L:= 4K_1+\max\{D,K_2(\epsilon/2),K_3,2\delta\}$.

\begin{enumerate}
\item[(1)] The Gromov product $(1\cdot a)_1 = 0$ for all $a\in \bC(S)$.
For all $y\in X_0$, there is $e_y\in gE$ such that $(e_y\cdot y)_1\geq L>0$, hence $1\not\in X_0$.

\item[(2)] If $y_i\map y\in\partial G$, then by the property of the Gromov product (see for example \cite[III.H 3.17(5)]{BH}),
$\lim\inf(x\cdot y_i)_1\geq (x\cdot y)_1-2\delta$.
This implies if $y\in X_{n+1}$, 
then for any sequence $y_i\map y$, 
all but finitely many $y_i$'s are in $X_n = S_1(gE,L(n+1))$ since $L>2\delta$.
Thus we have $X_{n+1}\cap (\mathcal{C}(S)\setminus X_n) = \emptyset$.

\item[(3)] Let $a\in X_{n+1}$,
then there exists $e_a\in gE$ such that $a\in S_1(e_a,L(n+2))$.
Let $b\in \mathcal{C}(S)\setminus X_n$, then for all $e\in gE$, we have $b\not\in S_1(e, L(n+1))$.
In particular $b\not\in S_1(e_a,L(n+1))$.
Then we consider a tree $T_1$ from Lemma \ref{lem.fourpt} that connects $\{1,b,a,e_a\}$.
Since $(a\cdot e_a)_1\geq L(n+2)$ and $(b\cdot e_a)_1<L(n+1)$, by (\ref{eq.t-gromov}),
the only possible combinatorial type of $T_1$ is 
the one we get by substituting $(x_1,x_2,x_3,x_4) = (1,b,a,e_a)$ in Figure \ref{fig.tree}.
Then we see that 
\begin{eqnarray*}
d_\mathcal{C}(a,b)&\geq& d_{T_1}(a,b)-K_1\geq d_{T_1}(p,q) - K_1 \\
&\geq& (a\cdot e_a)_1 - (b\cdot e_a)_1 - 4K_1\geq L-4K_1,
\end{eqnarray*}
where $p,q$ are the trivalent vertices as depicted in Figure \ref{fig.tree}.
Thus by the definition of $L$, we have $d_\mathcal{C}(a,b)\geq D$.

\item[(4)]
Let $x\in X_n\setminus X_{n+1}$ and $y\in X_{n+2}$.
Then there exists $e_y\in gE$ such that
$(e_y\cdot y)_1\geq L(n+3)$ and $(x\cdot e_y)<L(n+2)$.
Then, similarly as (3), by Lemma \ref{lem.fourpt}, we see that there is a tree $T_2$ with 
$(x_1,x_2,x_3,x_4) = (1,x,y,e_y)$ in Figure \ref{fig.tree}.
Then we have
\begin{eqnarray*}
(e_y\cdot y)_x&\geq& (e_y\cdot y)^{T_2}_x - K_1\geq d_{T_2}(p,q)-K_1\\
&\geq&(e_y\cdot y)_1-(e_y\cdot x)_1-4K_1\geq L-4K_1.
 \end{eqnarray*}
Hence $S_x(gE,L-4K_1)\supset X_{n+2}$.
This implies that 
$$\nu_x(X_{n+2}) \leq \nu_x(S_x(gE,L-4K_1)) = \nu(S_1(x^{-1}gE,L-4K_1)).$$
Then by Lemma \ref{lem.setshadowdecays},
we have $\nu_x(X_{n+2})\leq \nu(S_1(x^{-1}gE,L-4K_1))<1-\epsilon$ since $L-4K_1\geq K_3$.

\item[(5)]
Since $x\in X_n\setminus X_{n+1}$, there is $e\in gE$ such that $(x\cdot e)_1\geq L(n+1)$.
Let $y\not\in X_{n-1}$, which implies $(y\cdot e)_1<Ln$.
Similarly as (3) and (4), we have a tree $T_3$ for $(x_1, x_2, x_3, x_4) = (1,y,x,e)$ in Figure \ref{fig.tree}.
Then we have $$(1\cdot y)_x\geq (1\cdot y)_x^{T_3}-K_1\geq d_{T_3}(p,q) - 4K_1 \geq L -4K_1.$$
Thus, we see $y\in S_x(1, L-4K_1)$.
Hence we have $$\bC(S)\setminus X_{n-1}\subset S_x(1,L-4K_1).$$
Since we have chosen $L\geq4K_1+K_2(\epsilon/2)$, we see that by Lemma \ref{lem.pointshadowdecays}
$$\nu_x(\bC(S)\setminus X_{n-1})\leq \nu_x(S_x(1,L-4K_1)) = \nu(S_1(x^{-1},L-4K_1))<\epsilon/2.$$
\end{enumerate}
\end{proof}

\begin{proof}[Proof of Lemma \ref{lem.setshadow}]
By Lemma \ref{lem.translen}, we may suppose for some $L'>0$, 
$\tau(\omega_n)\geq L'n$ with exponentially high probability.
This implies that if $\omega_n\in gE$, then $\omega_n$ must be in $X_{\lfloor L'/L\rfloor n-1}$, 
where $X_i:=S_1(gE,L(i+1))$ as in Lemma \ref{lem.key}.
Therefore by Lemma \ref{lem.key}, we have
$$\mathbb{P}(\omega_n\in gE)\leq\mathbb{P}(\omega_n\in X_{\lfloor L'/L\rfloor n-1}) + 
\mathbb{P}(d_\mathcal{C}(\omega_n,\omega_0)<L'n)\leq Kc^n$$
 for some $K>0$ and $c<1$.
\end{proof}

\subsection{Proof of Theorem \ref{thm.notlift}}\label{sec.proof of 1.3}
We are now ready to prove Theorem \ref{thm.notlift}.
The proof goes similarly as the proof of Theorem \ref{thm.primitive}.
First we prepare an alternative of Lemma \ref{lem.align}.
\begin{lem}\label{lem.alignlift}
Fix $D_{0}, M>0$.
Then there is a constant $D_{1}>0$, $c_{1}<1$, $K>0$
such that the following holds.
Consider the collection of indices $a < b < c$
 with the following properties:
\begin{enumerate}
\item $d_\mathcal{C}(\omega_a,\omega_b)\geq D_{1}$,
\item $d_\mathcal{C}(\omega_b,\omega_c)\geq Mn$, and
\item there exists a covering $\pi:S\map S'$ such that $d_\mathcal{C}(\omega_i, \Pi(\mathcal{C}(S')))\leq D_{0}$ for all $i\in\{a,b,c\}$.
\end{enumerate}
Then the probability that this collection of indices is non-empty is at most $Kc_1^n$.
\end{lem}
\begin{proof}
The number of possible types of orbifolds which may be covered by $S$ is finite.
Furthermore, for each such an orbifold, there are only finitely many possible covering maps up to conjugacy.
This is because the number of subgroups of bounded index in a finitely generated group is finite.
Hence it suffices to fix a covering $\pi:S\map S'$ and consider only its conjugates.
Let $E:= \Pi(\mathcal{C}(S'))$ and $D$ be a constant that Lemma \ref{lem.finiteparallel} works for $\mathcal{P}(x,y,D)$ of any covering from $S$.

Suppose we have indices $a,b$ which satisfy condition (1). 
Then by Lemma \ref{lem.finiteparallel}, the cardinality of $\mathcal{P}(\omega_a,\omega_b,D)$ is universally bounded.
Hence to have a index $c$ which satisfies condition (2) and (3), the random walk that starts from $\omega_b$ must get into
$S_{\omega_b}(gE,Mn)$ for some $gE\in\mathcal{P}(\omega_a,\omega_b,D)$.
Since number of elements in $\mathcal{P}(\omega_a, \omega_b,D)$ is universally bounded,
and the number of possible choices of indices $a,b$ is of order $n^{2}$,
by Lemma \ref{lem.setshadow} we complete the proof.
\end{proof}

\begin{proof}[Proof of Theorem \ref{thm.notlift}]
For the readability of the proof we will not explicitly write the constants.
One can compute constants in a similar way to the proof of Theorem \ref{thm.primitive}.

Suppose $\omega_{n}$ is symmetric.
We may suppose $\omega_{n}$ is pseudo-Anosov.
Since the stable and unstable measured foliations of $\omega_{n}$ are in some $gE$, and $gE$ is quasi-convex,
any geodesic axis of $\omega_{n}$ fellow travels with $gE$.
By Lemma \ref{lem.proximal} and Proposition \ref{prop.fellowtravel},
we see that we can find some indices that satisfies the conditions of Lemma \ref{lem.alignlift} for suitable constants.
\end{proof}

\section{Applications}
\subsection{Cusped random mapping tori are non-arithmetic}\label{sec.arithmetic}
First, we recall the definition of non-compact arithmetic $3$-manifolds, 
see \cite{MR} for more details and properties of arithmetic $3$-manifolds.
Let $d$ be a positive square-free integer and $\mathcal{O}_d$ denote the ring of integers of $\mathbb{Q}(\sqrt{-d})$.
A {\em Bianchi group} is a subgroup of $\mathrm{PSL}(2,\mathbb{C})$ which is of the form 
$\mathrm{PSL}(2,\mathcal{O}_d)$.
One can show that every Bianchi group is a lattice.
The quotient $\mathbb{H}^3/\mathrm{PSL}(2,\mathcal{O}_d)$ is called a {\em Bianchi orbifold}, where
$\mathbb{H}^3$ is the hyperbolic 3-space.
A non-compact hyperbolic $3$-manifold $M = \mathbb{H}^3/\Gamma$ of finite volume is {\em arithmetic}
if a conjugate of $\Gamma$ in $\mathrm{PSL}(2,\mathbb{C})$ is commensurable to some 
Bianchi group $\mathrm{PSL}(2,\mathcal{O}_d)$.
Recall that two subgroups of $\mathrm{PSL}(2,\mathbb{C})$
 are said to be commensurable if their intersection is a finite index subgroup in both.
Let $S$ be an orientable surface of finite type with at least one puncture.
For $\phi\in\mathrm{Mod}(S)$, the mapping torus $M(S,\phi)$ is defined by
$$M(S,\phi) = S\times [0,1]/(x,1)\sim (\phi(x),0).$$
Two mapping tori $M(S,\phi_1)$ and $M(S,\phi_2)$ are said to be {\em cyclic commensurable} if
there exists $k_1,k_2\in\mathbb{Z}\setminus\{0\}$ such that $M(S,\phi_1^{k_1}) = M(S,\phi_2^{k_2})$.
Bowditch-Maclachlan-Reid proved the following theorem.
\begin{thm}[{\cite[Theorem 4.2]{BMR}}]\label{thm.bmr}
Let $S$ be an orientable surface of finite type with at least one puncture.
There are at most finitely many cyclic commensurability classes of arithmetic mapping tori with fiber $S$.
\end{thm}
Now we are in a position to prove Theorem \ref{thm.arithmetic}.
\begin{proof}[Proof of Theorem \ref{thm.arithmetic}]
Note that if two mapping classes give rise to cyclic commensurable mapping tori, then they are fibered commensurable.
By Theorem \ref{thm.minimal}, 
it suffices to discuss minimal mapping classes in their commensurability classes.
The uniqueness of the minimal element (Theorem \ref{thm.uniqueminimal}) implies that
two minimal mapping classes give rise to cyclic commensurable mapping tori if and only if they are conjugate.
Hence there are at most finitely many conjugacy classes of minimal elements that give arithmetic mapping tori
by Theorem \ref{thm.bmr}.
Hence there is an upper bound of the translation length for minimal mapping classes to have arithmetic mapping tori.
Then Lemma \ref{lem.translen} applies to complete the proof.
\end{proof}

\begin{rmk}
For $S$ closed,
one can prove similar statement as Theorem \ref{thm.bmr} with upper bound for the degree of the invariant trace fields, 
see \cite[Corollary 4.4.]{BMR}.
For $S$ closed, we do not know if the set of  a random mapping classes with arithmetic mapping tori is exponentially small or not.

\end{rmk}
\subsection{Closed random mapping tori are asymmetric}\label{sec.asymmetric}
It is well known that the isometry group of any closed hyperbolic 3-manifold is finite.
A closed hyperbolic 3-manifold is called {\em asymmetric} if the isometric group is trivial.
As a corollary of Theorem \ref{thm.minimal} and the work of Bachman-Schleimer \cite{BS}, 
we have the following.
\begin{thm}
Let $\mu$ be a probability measure on $G$ which satisfies Condition \ref{condi.munotlift}.
Then the set of mapping classes with asymmetric mapping tori is exponentially large with respect to $\mu$.
\end{thm}
\begin{proof}
By Lemma \ref{lem.translen}, the translation length $\tau(\omega_{n})$ grows linearly with $n$ with exponentially high probability.
By the work of Bachman-Schleimer \cite[Theorem 3.1]{BS}, 
we see that if the translation distance of $\tau(\phi)$ of $\phi$ is greater than $-\chi(S)$
then any isometry of $M(S,\phi)$ must maps each fiber to a fiber.
Theorem \ref{thm.minimal} implies that the probability that $M(S,\omega_{n})$ has isometry $h$ which maps each fiber to a fiber 
and the quotient by $\langle h\rangle$ is a $2$-orbifold bundle over the circle decays exponentially.
The only case remained is when $M(S,\omega_{n})$ admits an isometry of type $(x,t)\mapsto(\beta x,1-t)$ for some involution $\beta:S\rightarrow S$, in other words when $M(S,\omega_{n})$ admits a quotient which is a $2$-orbifold bundle over the $1$-orbifold $S^{1}/\mathbb{Z}_{2}$.
Note that we may suppose $\omega_{n}$ is pseudo-Anosov.
In this case we have $\beta\omega_{n}\beta = \omega_{n}^{-1}$ and especially $\beta$ permutes the elements in $\text{Fix}(\omega_{n})\subset\PMF(S)$.
Hence around geodesic axes of $\omega_{n}$ in the curve complex $\mathcal{C}(S)$, $\beta$ coarsely acts as a reflection.
Then by taking conjugate by $\omega_{n}^{k}$ for some $k\in\mathbb{Z}$ if necessary, we may suppose coarse fixed points of $\beta$ are on the $\gamma_{2}$ of the decomposition of $[\omega_{0},\omega_{n}]$ from Proposition \ref{prop.fellowtravel}.
Then with Remark \ref{rmk.anti-align}, we have desired conclusion by a similar argument to the proof of Theorem \ref{thm.primitive}.
\end{proof}

\section*{Acknowledgments}
The author would like to thank Ingrid Irmer, Joseph Maher, Makoto Sakuma and Giulio Tiozzo for helpful conversations. 
He would especially like to thank Joseph Maher for suggesting to use the work \cite{CM} to prove Theorem \ref{thm.primitive}.
He started this work when he was in  ICERM, Brown University.
Thanks also goes to  ICERM and JSPS for supporting the visit.
He would also like to thank Brian Bowditch for bringing the paper \cite{BMR} in his attention.
An earlier version of this paper contained a gap in the proof of Theorem \ref{thm.notlift}.
The author would like to thank a referee for pointing out the gap.
This work was partially supported by JSPS Research Fellowship for Young Scientists.

\end{document}